\documentclass[12pt,reqno]{article}

\usepackage[usenames]{color}
\usepackage[english]{babel}
\usepackage{algorithm}
\usepackage{algorithmic}
\usepackage{rotating}
\usepackage{amsmath}
\usepackage{amsthm}

\definecolor{webgreen}{rgb}{0,.5,0}
\definecolor{webbrown}{rgb}{.6,0,0}
\usepackage[colorlinks=true,
linkcolor=webgreen,
filecolor=webbrown,
citecolor=webgreen]{hyperref}

\newcommand{\seqnum}[1]{\href{http://oeis.org/#1}{\underline{#1}}}

\theoremstyle{plain}
\newtheorem{theorem}{Theorem}
\newtheorem{corollary}[theorem]{Corollary}
\newtheorem{lemma}[theorem]{Lemma}

\theoremstyle{definition}

\newtheorem{example}[theorem]{Example}

\begin{document}

\begin{center}
\vskip 1cm{\LARGE\bf 
A meet-in-the-middle algorithm for finding \\
\vskip .1in
extremal restricted additive 2-bases
}
\vskip 1cm
\large
Jukka Kohonen \\
Department of Mathematics and Statistics\\
P. O. Box 68\\
FI-00014 University of Helsinki\\
Finland \\
\href{mailto:jukka.kohonen@helsinki.fi}{\tt jukka.kohonen@helsinki.fi}
\end{center}

\begin{abstract}
An additive 2-basis with range $n$ is {\em restricted} if its largest
element is $n/2$.  Among the restricted 2-bases of given length $k$,
the ones that have the greatest range are {\em extremal restricted}.
We describe an algorithm that finds the extremal restricted 2-bases of
a given length, and we list them for lengths up to $k=41$.
\end{abstract}

%%%%%%%%%%%%%%%%%%%%%%%%%%%%%%%%%%%%%%%%%%%%%%%%%%%%%%%%%%%%%%%%%%%%%%

\section{Introduction}
Let $n$ be a positive integer.  An {\bf additive 2-basis for~$n$}, or
more briefly a {\bf basis for~$n$}, is a set of integers $A_k =
\{0=a_0 < a_1 < \cdots < a_k\}$ such that every integer in $[0,n]$ is
the sum of two of its elements, not necessarily distinct.  The {\bf
  length} of the basis is $k$.  The largest possible~$n$ for a basis
$A_k$ is its {\bf range} and denoted $n_2(A_k)$.  The maximum range
among all bases of length~$k$ is $n_2(k)$, and a basis that attains this
maximum is {\bf extremal}.

A basis $A_k$ is {\bf admissible} if $n_2(A_k) \ge a_k$, {\bf
  restricted} if $n_2(A_k) \ge 2a_k$, and {\bf symmetric} if $a_i +
a_{k-i} = a_k$ for all $0\le i\le k$.  Since for any basis $n_2(A_k) \le
2a_k$, a restricted basis has in fact $n_2(A_k) = 2a_k$ exactly.

The maximum range among restricted bases is called the {\bf extremal
  restricted range} and denoted $n_2^*(k)$, and an {\bf extremal
  restricted basis} is one that attains this maximum.  For many values
of $k$, at least some of the extremal bases are restricted, so that
$n_2^*(k)=n_2(k)$.  This is not always true: a counterexample is $k=10$,
where $n_2^*(10)=44$, but $n_2(10)=46$ (see Table~\ref{table:k10}).

Similarly, the maximum range among symmetric bases can be called the
{\bf extremal symmetric range}, and a basis that attains this maximum
can be called {\bf extremal symmetric basis}.  Extremal symmetric
bases are known up to $k=30$ due to Mossige \cite{mossige1981}.

\begin{table}[bt]
\begin{center}
\begin{tabular}{ll|l}
range & notes & basis \\
\hline
44 & R,A & 0 1 2 3 7 11 15 17 20 21 22 $\dagger$ \\
44 & R,S & 0 1 2 3 7 11 15 19 20 21 22 \\
44 & R,S & 0 1 2 5 7 11 15 17 20 21 22 \\
44 & R,A & 0 1 2 5 7 11 15 19 20 21 22 $\dagger$ \\
44 & R,S & 0 1 2 5 8 11 14 17 20 21 22 \\
44 & R,A & 0 1 3 4 6 11 13 18 19 21 22 $\ddagger$ \\
44 & R,S & 0 1 3 4 9 11 13 18 19 21 22 \\
44 & R,A & 0 1 3 4 9 11 16 18 19 21 22 $\ddagger$ \\
\hline
46 & NR,A & 0 1 2 3 7 11 15 19 21 22 24 \\
46 & NR,A & 0 1 2 5 7 11 15 19 21 22 24 \\
\hline
\end{tabular}
\caption{With length $k=10$, there are eight extremal restricted bases
  and two extremal nonrestricted bases, listed by Wagstaff
  \cite{wagstaff1979}.  The two bases marked with $\dagger$ are mirror
  images of each other; similarly the two bases marked with
  $\ddagger$.  Notes: R = restricted, NR = nonrestricted; S =
  symmetric, A = asymmetric.}
\label{table:k10}
\end{center}
\end{table}

All extremal bases of lengths $k\le 24$ are currently known
\cite{kohonen2014}.  Interestingly, all of them are either {\em both}
symmetric and restricted, or {\em neither}.  Three questions arise
naturally:
\begin{enumerate}
\item If an extremal basis is symmetric, is it necessarily restricted?
\item If an extremal basis is restricted, is it necessarily symmetric?
\item Does every $k \ge 15$ have an extremal basis that is symmetric?
\end{enumerate}

The first question is answered affirmatively by a theorem of Rohrbach
\cite{rohrbach1937}.  The second question was posed by Riddell and
Chan \cite[p.~631]{riddell1978}, but to our knowledge has not been
answered in general; with $k\le 24$ the answer is yes.

The third question has appeared in a stronger form: it was suggested
that all extremal bases with $k \ge 15$ might be symmetric
\cite{challis1993}.  This was later disproven by Challis and Robinson,
since $k=21$ has three extremal bases: one symmetric, two asymmetric
\cite{challis2010}.  The question remains whether every $k \ge 15$ has
at least one extremal basis that is symmetric.

In this work we describe an efficient algorithm for finding all
extremal restricted bases of a given length $k$.  The algorithm is
based on the idea that a restricted basis can be constructed by
concatenating two shorter admissible bases, one of them as a mirror
image.  With this method we have computed all extremal restricted
bases of lengths $k \le 41$.

Note that we have included $a_0=0$ in a basis, similarly to Wagstaff
\cite{wagstaff1979}.  If $0$ is excluded, the equivalent condition is
that every integer in $[0,n]$ is the sum of {\em at most} two elements
of the basis \cite[p.~3.1]{selmer1986}.  Excluding the zero is perhaps
more usual in current literature, but including it is more convenient
for our purposes.  The zero element is not counted in the length of
a basis.

\section{Related work}
Our search algorithm builds on a combination of existing ideas.
Rohrbach discusses symmetric bases, and the proof of his Satz 1 is
based on the observation that if a basis is mirrored from $a_k$, then
its pairwise sums are mirrored from $2a_k$ \cite{rohrbach1937}.  We
shall exploit a generalization of this for asymmetric restricted
bases.

Riddell and Chan discuss the connection between symmetric and
restricted bases \cite{riddell1978}.  Mossige notes that symmetric
bases $A_k$ can be efficiently searched by scanning through admissible
bases of length $A_{\lceil k/2 \rceil}$ \cite{mossige1981}.  For
symmetric bases this is sufficient; the second half of a symmetric
basis is a mirror image of the first half, and then Rohrbach's theorem
ensures that the constructed set $A_k$ is a basis for $2a_k$.  For
asymmetric restricted bases, a similar search can be conducted
separately for the two halves of the basis (prefix and suffix).
However, since Rohrbach's theorem does not apply to asymmetric bases,
the construction does not automatically yield a basis for $2a_k$.
This must be checked separately.

The final ingredient is the ``gaps test'' by Challis
\cite{challis1993}.  Based on a simple combinatorial argument, it
prunes the search tree of admissible bases, if they are required to
have a range of at least a given target value $T$.  In
section~\ref{section:prefix} we shall prove lower bounds for the
ranges of the prefix and the mirrored suffix.  With these lower bounds
the gaps test prunes the search tree very efficiently.

\section{Definitions and initial results}
If $A$ and $B$ are sets of integers, we define
\begin{equation*}
  A+B := \{a+b : a \in A, b \in B\},
\end{equation*}
and if $b$ is an integer, we define the {\bf mirror image} of $A$ with respect to $b$ as
\begin{equation*}
  b-A := \{b-a : a \in A\}.
\end{equation*}
The set of integers {\bf generated by} $A$ is
\begin{equation*}
  2A := A+A = \{a+a' : a,a' \in A\}.
\end{equation*}
It is straightforward to verify that
\begin{equation*}
  2(b-A) = 2b - 2A.
\end{equation*}
By $[c,d]$ we denote the consecutive integers $\{c, c+1, \ldots, d\}$.
Now the condition that $A$ is a basis for $n$ is succintly stated as
follows:
\begin{equation*}
  2A \supseteq [0,n].
\end{equation*}
If $A_k = \{a_0 < \cdots < a_k\}$ is a basis and $i<k$, then $A_i =
\{a_0, \ldots, a_i\}$ is a {\bf partial basis}.  We state without
proof three easy observations (see \cite{challis1993} and
\cite{selmer1986}):

\begin{lemma}
  If a basis is restricted, then it is admissible.
  \label{lemma:restadm}
\end{lemma}

\begin{lemma}
  If a basis is extremal, then it is admissible.
  \label{lemma:extadm}
\end{lemma}

\begin{lemma}
  If a basis $A_k$ is admissible, then for all $i<k$ the partial basis
  $A_i$ is admissible, and $a_{i+1} \le n_2(A_i)+1$.
  \label{lemma:partial}
\end{lemma}

The first question posed in the introduction is now answered by the
following theorem, essentially the same as Rohrbach's Satz 1
\cite[p.~4]{rohrbach1937}.

\begin{theorem}
  If $A_k$ is an extremal basis and it is symmetric, then it is
  restricted.
  \label{theorem:symmrest}
\end{theorem}
\begin{proof}
  Let $a_k = \max\{A_k\}$.  By Lemma~\ref{lemma:extadm}, $A_k$ is
  admissible; thus $2A_k \supseteq [0,a_k]$.  By symmetry $A_k =
  a_k-A_k$, thus
  \begin{equation*}
    2A_k = 2(a_k-A_k) = 2a_k - 2A_k \supseteq [a_k, 2a_k].
  \end{equation*}
  Combining the above observations we have $2A_k \supseteq [0,2a_k]$,
  thus $A_k$ is restricted.
\end{proof}

Note that if $A_k$ is a restricted basis with range $n$, then its
largest element is exactly $a_k=n/2$.  Exploiting the idea of
mirroring from the largest element we obtain the following theorem.

\begin{theorem}
  If $A_k$ is a restricted basis with range $n$, then $a_k-A_k$ is
  also a restricted basis for $n$.
  \label{theorem:mirror}
\end{theorem}
\begin{proof}
  Since $A_k$ is a basis for $n$, it follows that $2A_k \supseteq
  [0,n]$.  Now
  \begin{equation*}
    2(a_k-A_k) = 2a_k - 2A_k = n - 2A_k \supseteq [0,n],
  \end{equation*}
  thus $a_k-A_k$ is a basis for $n$.  Its largest element is $a_k-0 =
  a_k = n/2$, thus it is restricted.
\end{proof}

Theorem~\ref{theorem:mirror} implies that asymmetric restricted bases
always form pairs that are mirror images of each other.  Two such
pairs are seen in Table~\ref{table:k10}.

\section{Prefix and suffix of a restricted basis}
\label{section:prefix}

Let $A_k$ be a restricted basis with range $n$ and length $k \ge 3$.
Then by Theorem~\ref{theorem:mirror} the mirror image $B_k = a_k-A_k$
is also a restricted basis with range $n$.  Choose now an arbitrary
{\bf pivot index} $i$ such that $0<i<k-1$.  Split $A_k$ into a {\bf
  prefix} $A_i = \{a_0 < \cdots < a_i\}$ and a {\bf suffix} $R =
\{a_{i+1} < \cdots < a_k\}$.  The prefix is a partial basis of $A_k$.
The suffix can be mirrored from $a_k$ to obtain another basis
\begin{equation*}
B_j = a_k - R = \{b_0 < \cdots < b_j\},
\end{equation*}
where $j=k-1-i$, and $b_h = a_k - a_{k-h}$ for all $0\le h \le j$.
Now $B_j$ is a partial basis of $B_k$.

By Lemma~\ref{lemma:restadm} both $A_k$ and $B_k$ are admissible, and
then by Lemma~\ref{lemma:partial}
\begin{align*}
  n_2(A_i) &\ge a_{i+1}-1,\\
  n_2(B_j) &\ge b_{j+1}-1.
\end{align*}
We have now lower bounds for the ranges $n_2(A_i)$ and $n_2(B_j)$, but the
bounds depend on $a_{i+1}$ and $b_{j+1}$.  However, these values can
further be bounded from below:
\begin{align*}
  a_{i+1} &= a_k - b_j \ge a_k - (n_2(B_{j-1})+1) \ge a_k - n_2(j-1) - 1, \\
  b_{j+1} &= a_k - a_i \ge a_k - (n_2(A_{i-1})+1) \ge a_k - n_2(i-1) - 1,
\end{align*}
where $n_2(i-1)$ and $n_2(j-1)$ are the maximum ranges of bases of lengths
$i-1$ and $j-1$, respectively.  These maximum ranges are currently
known up to length $24$.

Combining these bounds we can state a necessary condition for
$A_k$ being a restricted basis with range $n$.

\begin{theorem}
  If $A_k$ is a restricted basis with range $n$, and $i$ is
  an index such that $0<i<k-1$, and $i+j=k-1$, then:
  \begin{enumerate}
  \item The prefix $A_i$ is an admissible basis such that
    $n_2(A_i) \ge a_k - n_2(j-1) - 2$.
  \item The mirrored suffix $B_j = a_k-\{a_{i+1},\ldots,a_k\}$ is an
    admissible basis such that $n_2(B_j) \ge a_k - n_2(i-1) - 2$.
  \end{enumerate}
  \label{theorem:main}
\end{theorem}

\begin{example}
Let $k=10$ and $n=44$, and choose $i=5$ (thus $j=4$).  If $A_{10}$ is
a restricted basis for $44$, then $a_{10}=22$.

Since $n_2(3)=8$, $b_4$ cannot be greater than $8+1=9$; in other words,
$a_6 = 22-b_4$ cannot be smaller than $22-9 = 13$; thus $n_2(A_5) \ge
12$.

Similarly, since $n_2(4)=12$, $a_5$ cannot be greater than $12+1=13$; in
other words, $b_5 = 22-a_5$ cannot be smaller than $22-13=9$; thus
$n_2(B_4) \ge 8$.

The lower bounds are the ones given by Theorem~\ref{theorem:main}.
Consider now a restricted basis $A_{10}$ and its mirror image
$B_{10}$, shown in right-to-left order:
\begin{equation*}
  \begin{array}{lrrrrrrrrrrrl}
    & \multicolumn{6}{c}{A_5}   & \multicolumn{5}{c}{R} \\
    & \multispan6\mathstrut\downbracefill & \multispan5\mathstrut\downbracefill \\
    A_{10} = & 0&  1& 3& 4& 6&11&13&18&19&21&22 \\
            &22&21&19&18&16&11& 9& 6& 3& 1& 0  & = B_{10}\\
    \multispan7               & \multispan5\mathstrut\upbracefill \\
    \multispan7               & \multicolumn{5}{c}{B_4} \\
  \end{array}
\end{equation*}
The prefix $A_5$ has range $12$, and the mirrored suffix $B_4$ has
range $10$.  Both ranges are within the bounds required by
Theorem~\ref{theorem:main}.
\end{example}

The second part of Theorem~\ref{theorem:main} also provides an upper
bound for the range of a restricted basis:
\begin{align*}
  n_2(A_k) = 2a_k &\le 2n_2(B_j) + 2n_2(i-1) + 4 \\
                &\le 2n_2(j) + 2n_2(i-1) + 4.
\end{align*}
Choosing $i = \lfloor k/2 \rfloor$ this yields the following bounds,
for even and odd values of~$k$.
\begin{corollary}
If $r>1$ is an integer, then
\begin{align*}
  n_2^*(2r)   &\le 4n_2(r-1) + 4,   \\
  n_2^*(2r+1) &\le 2n_2(r-1) + 2n_2(r) + 4.
\end{align*}
\label{cor:upper}
\end{corollary}

\section{Search algorithm}
Suppose that $k$ and $n$ are given, and the task is to enumerate every
restricted basis of length $k$ and range $n$ (if any such bases
exist).  Choose a pivot index $i$, for example $i = \lfloor k/2
\rfloor$.

A straightforward method would be to enumerate all admissible prefix
bases $A_i$, all admissible mirrored suffix bases $B_j$, and for each
pair $(A_i, B_j)$ check whether $A_i \cup (n/2 - B_j)$ happens to be a
basis for $n$, that is, whether it generates all integers in $[0,n]$.
For large $k$ this is not feasible, as the number of admissible bases
of length $i$ increases rapidly (see \seqnum{A167809} in \cite{oeis}).

However, Theorem~\ref{theorem:main} gives definite lower bounds for
the ranges of the prefix $A_i$ and the mirrored suffix $B_j$.  Thus
only a tiny fraction of all admissible prefixes and mirrored suffixes
need to be considered, as seen in the following example.

\begin{example}
Let $k=25$ and $n=228$.  We want to know whether there are any
restricted bases with these values, and to list them if there are.
Choose $i=12$ (thus $j=12$).  The last element of a restricted basis
must be $a_k = n/2 = 114$.  There are $15\;752\;080$ admissible bases
of length 12, but we only need to consider the prefixes $A_i$ such
that $n_2(A_i) \ge 114-n_2(11)-2 = 58$; there are only $187$ such
prefixes.
\end{example}

Admissible bases with a given length and a given minimum range can be
enumerated with the algorithm (``K-program'') described by Challis
\cite{challis1993}.  Combining these ingredients we obtain
Algorithm~\ref{alg:mim}, which enumerates all restricted bases of
given length $k$ and range $n$.

\begin{algorithm}
\caption{List restricted bases of length $k$ and range $n$}
\begin{algorithmic}[1]
\REQUIRE $k \ge 3$
\STATE $i \leftarrow \lfloor k/2 \rfloor$ \COMMENT{Choose pivot index}
\STATE $j \leftarrow k-i-1$
\STATE $n_a \leftarrow n_2(i-1)$ \COMMENT{Lookup from \seqnum{A001212}}
\STATE $n_b \leftarrow n_2(j-1)$ \COMMENT{Lookup from \seqnum{A001212}}
\STATE ${\cal A} \leftarrow \{A_i : n_2(A_i) \ge n/2 - n_b - 2 \}$   \COMMENT{List prefixes with Challis algorithm \cite{challis1993}}
\STATE ${\cal B} \leftarrow \{B_j : n_2(B_j) \ge n/2 - n_a - 2 \}$  \COMMENT{List mirrored suffixes with Challis algorithm}
\FORALL{$A_i \in {\cal A}$}
  \FORALL{$B_j \in {\cal B}$}
     \STATE $R \leftarrow n/2 - B_j$ \COMMENT{Mirror from $n/2$}
     \STATE $A_k \leftarrow A_i \cup R$ \COMMENT{Concatenate}
     \IF[Generate pairwise sums and check range]{$2A_k \supseteq [0,n]$}
       \PRINT{$A_k$}
     \ENDIF
  \ENDFOR
\ENDFOR
\end{algorithmic}
\label{alg:mim}
\end{algorithm}

If $n_2^*(k)$ is not known, Algorithm~\ref{alg:mim} can be run with
different values of $n$, starting from the upper bound for $n_2^*(k)$
provided by Corollary~\ref{cor:upper}.  If no solutions are found, $n$
is then decreased in steps of $2$, until for some $n$ there are
solutions.  Only even values of $n$ need to be considered, since the
range of a restricted basis is always even.

\begin{example}
  Let $k=25$.  By Corollary~\ref{cor:upper}, $n_2^*(25) \le 240$.  For
  $n=240$ the search algorithm finds no solutions.  Then $n$ is
  reduced in steps of $2$, until for $n=228$ the algorithm returns one
  solution:
  \begin{align*}
    A_{25} = \{&0,1,3,4,6,10,13,15,21,29,37,45,53,\\
              &61,69,77,85,93,99,101,104,108,110,111,113,114 \}
  \end{align*}
  By construction, this is an extremal restricted basis,
  so now we know that $n_2^*(25) = 228$.
\end{example}

\section{Results}
Using the search algorithm described in the previous section, we
performed an exhaustive search for extremal restricted bases of
lengths $k=25,\ldots,41$.  The bases are listed in
Table~\ref{table:newrestricted}.  For ease of reference, previously
known extremal restricted bases of lengths $k=1,\ldots,24$ are listed
in Table~\ref{table:oldrestricted}.

\begin{sidewaystable}[p]
\begin{center}
\small
\setlength{\tabcolsep}{4pt}
\begin{tabular}{lll|rrrrrrrrrrrrrrrrrrrrrrrrrrrrrrrrrrr}
$k$ & $n_2^*(k)$ & & \multicolumn{5}{l}{basis} \\
\hline
  1 &   2 & S&   0&   1\\
\hline
  2 &   4 & S&   0&   1&   2\\
\hline
  3 &   8 & S&   0&   1&   3&   4\\
\hline
  4 &  12 & S&   0&   1&   3&   5&   6\\
\hline
  5 &  16 & S&   0&   1&   3&   5&   7&   8\\
\hline
  6 &  20 & S&   0&   1&   2&   5&   8&   9&  10\\
  6 &  20 & S&   0&   1&   3&   5&   7&   9&  10\\
\hline
  7 &  26 & S&   0&   1&   2&   5&   8&  11&  12&  13\\
  7 &  26 & S&   0&   1&   3&   4&   9&  10&  12&  13\\
\hline
  8 &  32 & S&   0&   1&   2&   5&   8&  11&  14&  15&  16\\
\hline
  9 &  40 & S&   0&   1&   3&   4&   9&  11&  16&  17&  19&  20\\
\hline
 10 &  44 &  &   \multicolumn{8}{l}{see Table~\ref{table:k10}} \\
\hline
 11 &  54 & S&   0&   1&   3&   4&   9&  11&  16&  18&  23&  24&  26&  27\\
 11 &  54 & S&   0&   1&   3&   5&   6&  13&  14&  21&  22&  24&  26&  27\\
\hline
 12 &  64 & S&   0&   1&   3&   4&   9&  11&  16&  21&  23&  28&  29&  31&  32\\
\hline
 13 &  72 & S&   0&   1&   3&   4&   9&  11&  16&  20&  25&  27&  32&  33&  35&  36\\
\hline
 14 &  80 & S&   0&   1&   3&   4&   5&   8& $\cdots$ &  +6 & $\cdots$&  32&  35&  36&  37&  39&  40\\
\hline
 15 &  92 & S&   0&   1&   3&   4&   5&   8& $\cdots$ &  +6 & $\cdots$&  38&  41&  42&  43&  45&  46\\
\hline
 16 & 104 & S&   0&   1&   3&   4&   5&   8& $\cdots$ &  +6 & $\cdots$&  44&  47&  48&  49&  51&  52\\
\hline
 17 & 116 & S&   0&   1&   3&   4&   5&   8& $\cdots$ &  +6 & $\cdots$&  50&  53&  54&  55&  57&  58\\
\hline
 18 & 128 & S&   0&   1&   3&   4&   5&   8& $\cdots$ &  +6 & $\cdots$&  56&  59&  60&  61&  63&  64\\
\hline
 19 & 140 & S&   0&   1&   3&   4&   5&   8& $\cdots$ &  +6 & $\cdots$&  62&  65&  66&  67&  69&  70\\
\hline
 20 & 152 & S&   0&   1&   3&   4&   5&   8& $\cdots$ &  +6 & $\cdots$&  68&  71&  72&  73&  75&  76\\
\hline
 21 & 164 & S&   0&   1&   3&   4&   5&   8& $\cdots$ &  +6 & $\cdots$&  74&  77&  78&  79&  81&  82\\
 21 & 164 & S&   0&   1&   3&   4&   6&  10&  13&  15&  21& $\cdots$ &  +8 & $\cdots$&  61&  67&  69&  72&  76&  78&  79&  81&  82\\
\hline
 22 & 180 & S&   0&   1&   3&   4&   6&  10&  13&  15&  21& $\cdots$ &  +8 & $\cdots$&  69&  75&  77&  80&  84&  86&  87&  89&  90\\
\hline
 23 & 196 & S&   0&   1&   3&   4&   6&  10&  13&  15&  21& $\cdots$ &  +8 & $\cdots$&  77&  83&  85&  88&  92&  94&  95&  97&  98\\
\hline
 24 & 212 & S&   0&   1&   3&   4&   6&  10&  13&  15&  21& $\cdots$ &  +8 & $\cdots$&  85&  91&  93&  96& 100& 102& 103& 105& 106\\
\end{tabular}
\caption{Extremal restricted bases of lengths $k=1,\ldots,24$.  S =
  symmetric, A = asymmetric.  $+c$ indicates several elements with a
  repeated difference of $c$.}
\label{table:oldrestricted}
\end{center}
\end{sidewaystable}

\begin{sidewaystable}[p]
\begin{center}
\small
\setlength{\tabcolsep}{2pt}
\begin{tabular}{lll|rrrrrrrrrrrrrrrrrrrrrrrrrrrrrrrrrrr}
$k$ & $n_2^*(k)$ & & \multicolumn{5}{l}{basis} \\
\hline
 25 & 228 & S&   0&   1&   3&   4&   6&  10&  13&  15&  21& $\cdots$ & +8 & $\cdots$&  93&  99& 101& 104& 108& 110& 111& 113& 114\\
\hline
 26 & 244 & S&   0&   1&   3&   4&   6&  10&  13&  15&  21& $\cdots$ & +8 & $\cdots$& 101& 107& 109& 112& 116& 118& 119& 121& 122\\
 26 & 244 & S&   0&   1&   3&   4&   5&   8&  11&  15&  16& $\cdots$ & +9 & $\cdots$& 106& 107& 111& 114& 117& 118& 119& 121& 122\\
\hline
 27 & 262 & S&   0&   1&   3&   4&   5&   8&  11&  15&  16& $\cdots$ & +9 & $\cdots$& 115& 116& 120& 123& 126& 127& 128& 130& 131\\
\hline
 28 & 280 & S&   0&   1&   3&   4&   5&   8&  11&  15&  16& $\cdots$ & +9 & $\cdots$& 124& 125& 129& 132& 135& 136& 137& 139& 140\\
\hline
 29 & 298 & S&   0&   1&   3&   4&   5&   8&  11&  15&  16& $\cdots$ & +9 & $\cdots$& 133& 134& 138& 141& 144& 145& 146& 148& 149\\
\hline
 30 & 316 & S&   0&   1&   3&   4&   5&   8&  11&  15&  16&  25&  34& $\cdots$ &  +9 & $\cdots$& 124& 133& 142& 143& 147& 150& 153& 154& 155& 157& 158\\
 30 & 316 & S&   0&   1&   2&   5&   6&   8&  13&  14&  17&  19&  29& $\cdots$ & +10 & $\cdots$& 129& 139& 141& 144& 145& 150& 152& 153& 156& 157& 158\\
 30 & 316 & A&   0&   1&   2&   5&   6&   8&  13&  14&  17&  19&  29& $\cdots$ & +10 & $\cdots$& 129& 133& 139& 141& 146& 150& 152& 154& 155& 157& 158\\
 30 & 316 & A&   0&   1&   3&   4&   6&   8&  12&  17&  19&  25&  29& $\cdots$ & +10 & $\cdots$& 129& 139& 141& 144& 145& 150& 152& 153& 156& 157& 158\\
 30 & 316 & S&   0&   1&   3&   4&   6&   8&  12&  17&  19&  25&  29& $\cdots$ & +10 & $\cdots$& 129& 133& 139& 141& 146& 150& 152& 154& 155& 157& 158\\
 30 & 316 & S&   0&   1&   3&   4&   7&   8&   9&  16&  17&  21&  24& $\cdots$ & +11 & $\cdots$& 134& 137& 141& 142& 149& 150& 151& 154& 155& 157& 158\\
\hline
 31 & 338 & S&   0&   1&   3&   4&   7&   8&   9&  16&  17&  21&  24& $\cdots$ & +11 & $\cdots$& 145& 148& 152& 153& 160& 161& 162& 165& 166& 168& 169\\
\hline
 32 & 360 & S&   0&   1&   3&   4&   7&   8&   9&  16&  17&  21&  24& $\cdots$ & +11 & $\cdots$& 156& 159& 163& 164& 171& 172& 173& 176& 177& 179& 180\\
\hline
 33 & 382 & S&   0&   1&   3&   4&   7&   8&   9&  16&  17&  21&  24& $\cdots$ & +11 & $\cdots$& 167& 170& 174& 175& 182& 183& 184& 187& 188& 190& 191\\
\hline
 34 & 404 & S&   0&   1&   3&   4&   7&   8&   9&  16&  17&  21&  24& $\cdots$ & +11 & $\cdots$& 178& 181& 185& 186& 193& 194& 195& 198& 199& 201& 202\\
\hline
 35 & 426 & S&   0&   1&   3&   4&   7&   8&   9&  16&  17&  21&  24& $\cdots$ & +11 & $\cdots$& 189& 192& 196& 197& 204& 205& 206& 209& 210& 212& 213\\
\hline
 36 & 448 & S&   0&   1&   3&   4&   7&   8&   9&  16&  17&  21&  24& $\cdots$ & +11 & $\cdots$& 200& 203& 207& 208& 215& 216& 217& 220& 221& 223& 224\\
\hline
 37 & 470 & S&   0&   1&   3&   4&   7&   8&   9&  16&  17&  21&  24& $\cdots$ & +11 & $\cdots$& 211& 214& 218& 219& 226& 227& 228& 231& 232& 234& 235\\
\hline
 38 & 492 & S&   0&   1&   3&   4&   7&   8&   9&  16&  17&  21&  24& $\cdots$ & +11 & $\cdots$& 222& 225& 229& 230& 237& 238& 239& 242& 243& 245& 246\\
\hline
 39 & 514 & S&   0&   1&   3&   4&   7&   8&   9&  16&  17&  21&  24& $\cdots$ & +11 & $\cdots$& 233& 236& 240& 241& 248& 249& 250& 253& 254& 256& 257\\
\hline
 40 & 536 & S&   0&   1&   3&   4&   7&   8&   9&  16&  17&  21&  24&  35&  46& $\cdots$ & +11 & $\cdots$& 222& 233& 244& 247& 251& 252& 259& 260& 261& 264& 265& 267& 268\\
 40 & 536 & S&   0&   1&   2&   5&   7&  10&  11&  19&  21&  22&  25&  29&  30& $\cdots$ & +13 & $\cdots$& 238& 239& 243& 246& 247& 249& 257& 258& 261& 263& 266& 267& 268\\
\hline
 41 & 562 & S&   0&   1&   2&   5&   7&  10&  11&  19&  21&  22&  25&  29&  30& $\cdots$ & +13 & $\cdots$& 251& 252& 256& 259& 260& 262& 270& 271& 274& 276& 279& 280& 281\\
\end{tabular}
\caption{Extremal restricted bases of lengths $k=25,\ldots,41$.  S =
  symmetric, A = asymmetric.  $+c$ indicates several elements with a
  repeated difference of $c$.}
\label{table:newrestricted}
\end{center}
\end{sidewaystable}

For lengths $25,\ldots,29$ the extremal restricted bases are the
extremal symmetric bases listed by Mossige \cite{mossige1981}.  For
lengths $31,\ldots,41$ they equal the bases given by Challis and
Robinson's construction \cite[p.~6]{challis2010}.  Note that while the
aforementioned construction gives a lower bound for the extremal
restricted range, exhaustive search gives the exact range.

With $k=30$, there are six extremal restricted bases with range $316$.
Four of them are symmetric and were listed by Mossige, but two are
asymmetric.  This is perhaps unexpected, and shows that at least one
of the questions 2 and 3 stated in the introduction must be answered
negatively.  It is currently not known whether $n_2(30)$ is $316$ or
greater.
\begin{itemize}
\item If $n_2(30)=316$, then we have here two extremal bases that are
  restricted, but asymmetric; this would answer question 2 negatively.
\item If $n_2(30)>316$, then there must be some (currently unknown)
  nonrestricted bases with range greater than $316$, but they cannot
  be symmetric (for if they were, they would be restricted by
  Theorem~\ref{theorem:symmrest}).  This would answer question 3
  negatively.
\end{itemize}

As an example of the time requirement, with $k=41$ and $n=562$, the
algorithm generates $5\;514$ prefixes of length $20$ and range at
least $n/2-n_2(19)-2=139$.  These were enumerated in $120$ CPU hours
on parallel 2.6~GHz Intel Xeon processors, with a C++ implementation
of the Challis algorithm.  Since $41$ is odd, we have $j=i$, and the
mirrored suffixes are the same as the prefixes.  The concatenation
phase of the algorithm (lines 7 to 15) took 1.8 seconds with a Matlab
implementation.

\section{Discussion}
Restricted bases are an interesting class of additive bases for two
reasons.  On one hand, searching for the extremal solutions among
restricted bases is enormously faster than searching among all
additive bases, as illustrated in the previous sections.  This
efficiency stems from Theorem~\ref{theorem:mirror}, which places a
very strong constraint on any extremal restricted basis: that its
mirror image must also be a restricted basis (possibly different).
Thus restricted additive bases can be seen as a generalization of
symmetric additive bases.

On the other hand, among lengths $k=1,\ldots,24$, in almost every case
at least one of the extremal bases is restricted (with the sole
exception of $k=10$).  The reason for this is not known, and it is not
known whether this regularity continues for $k>24$.  The case of
$k=30$, discussed in the previous section, suggests that there may be
surprises waiting to be found.

For simplicity, we have always taken $i = \lfloor k/2 \rfloor$ in our
search algorithm.  Further research is needed to find the optimal
pivot index $i$ that minimizes the search work.

While Theorem~\ref{theorem:mirror} as such does not apply to
nonrestricted bases, it would be interesting to know if it could be
generalized in such a way that applies to them.  Such a generalization
might provide an improved search method for extremal additive bases in
the nonrestricted case.

\bigskip
\hrule
\bigskip
\noindent 2000 {\it Mathematics Subject Classification}:
Primary 11B13.

\noindent \emph{Keywords: } additive basis, restricted basis.

\bigskip
\hrule
\bigskip

\noindent (Concerned with sequences \seqnum{A001212},
\seqnum{A006638}, and \seqnum{A167809}.)

\bigskip
\hrule
\bigskip

\end{document}